\documentclass[vecarrow]{svmult}

\usepackage{makeidx}         
\usepackage{graphicx}        
\usepackage{psfrag}          
\usepackage{multicol}        
\usepackage{amsmath}
\usepackage{amssymb}
\usepackage{amsfonts}
\usepackage{latexsym}
\usepackage[T1]{fontenc}
\usepackage{subfigure}                      
\usepackage{float}
\usepackage{mathrsfs}

\begin{document}

\title{On the Spectrum of Volume Integral Operators in Acoustic Scattering}

\titlerunning{Volume Integral Operators in Acoustic Scattering}

\author{M. Costabel}

\authorrunning{M. Costabel}

\institute{IRMAR, Universit\'e de Rennes 1, France;
\texttt{martin.costabel@univ-rennes1.fr}}

\maketitle

\newcommand{\oH}{{\mathaccent'27 H}}
\newcommand{\grad}{\nabla}
\newcommand{\C}{\mathbb{C}}
\newcommand{\Id}{\mathbb{I}}
\newcommand{\R}{\mathbb{R}}
\newcommand{\Div}{\mathop{\mathrm{div}}}

\section{Volume Integral Equations in Acoustic Scattering}
Volume integral equations have been used as a theoretical tool in scattering theory for a long time. A classical application is an existence proof for the scattering problem based on the theory of Fredholm integral equations. This approach is described for acoustic and electromagnetic scattering in the books by Colton and Kress \cite{CoKr83,CoKr98} where volume integral equations appear under the name ``Lippmann-Schwinger equations''.

In electromagnetic scattering by penetrable objects, the volume integral equation (VIE) method has also been used for numerical computations. In particular the class of discretization methods known as ``discrete dipole approximation'' \cite{PuPe73,DrFl94} has become a standard tool in computational optics applied to atmospheric sciences, astrophysics and recently to nano-science under the keyword ``optical tweezers'', see the survey article \cite{YuHo07} and the literature quoted there. 
In sharp contrast to the abundance of articles by physicists describing and analyzing applications of the VIE method, the mathematical literature on the subject consists only of a few articles. An early spectral analysis of a VIE for magnetic problems was given in \cite{FriedmanPasciak1984}, and more recently \cite{Kirsch2007,KirschLechleiter2009} have found sufficient conditions for well-posedness of the VIE in electromagnetic and acoustic scattering with variable coefficients. In \cite{CoDarKo2010,CoDarSak2012}, we investigated the essential spectrum of the VIE in electromagnetic scattering under general conditions on the complex-valued coefficients, finding necessary and sufficient conditions for well-posedness in the sense of Fredholm in the physically relevant energy spaces. A detailed presentation of these results can be found in the thesis \cite{SaklyThesis2014}. Publications based on the thesis are in preparation.
Curiously, whereas the study of VIE in electromagnetic scattering has thus been completed as far as questions of Fredholm properties are concerned, the simpler case of acoustic scattering does not seem to have been covered in the same depth. It is the purpose of the present paper to close this gap.

The basic idea of the VIE method in scattering by a penetrable object is to consider the effect of the scatterer as a perturbation of a whole-space constant coefficient problem and to solve the latter by convolution with the whole-space fundamental solution. In the acoustic case, we consider the scalar linear elliptic equation
\begin{equation}
\label{e:scat}
\Div a(x) \grad u + k(x)^{2} u = f \qquad \mbox{ in }\R^{d}
\end{equation}
where we suppose that the (in general complex-valued) coefficients $a$ and $k$ are constant outside of a compact set:
$$
  a(x) \equiv 1,\quad  k(x)\equiv k\in\C \quad\mbox{ outside of the bounded domain } \Omega.
$$
and $f$ has compact support. We further assume that $u$ satisfies the outgoing Sommerfeld radiation condition.
It is well known that under very mild conditions on the regularity of the coefficients $a$ and $k$, there is at most one solution of this problem.

We then rewrite \eqref{e:scat} as a perturbed Helmholtz equation.
\begin{equation}
\label{e:perturb}
 (\Delta + k^{2}) u = f - \Div\alpha\grad u - \beta u
\end{equation}
with 
$$
 \alpha(x)= a(x)-1,\; \beta(x)=k(x)^{2}-k^{2}\,.
$$

Let now $G_{k}$ be the outgoing full-space fundamental solution of the Helm\-holtz equation, i.e. the unique distribution in $\R^{d}$ satisfying $(\Delta + k^{2}) G_{k}=-\delta$ and the Sommerfeld radiation condition. In dimension $d=3$, we have 
$$
  G_{k}(x) = \frac{e^{ik|x|}}{4\pi|x|}\,.
$$
We obtain the VIE from the following well known lemma.
\begin{lemma}
\label{l:conv}
Let $u$ be a distribution in $\R^{d}$ satisfying $-(\Delta + k^{2}) u = v$, where $v$ has compact support, and the Sommerfeld radiation condition. Then $u=G_{k}*v$, and if $v$ is an integrable function, the convolution can be written as an integral:
$$
  u(x) = \int G_{k}(x-y)\,v(y)\,dy \;.
$$
\end{lemma}
Applying this lemma to  \eqref{e:perturb}, we obtain the equation
$$
 u = -G_{k}*f + \Div G_{k}*(\alpha\grad u) + G_{k}*(\beta u),
$$
valid in the distributional sense on $\R^{d}$. This can be written as a VIE
\begin{equation}
\label{e:vie}
u(x) - \Div\int_{\Omega}G_{k}(x-y)\alpha(y)\grad u(y)\, dy
      -   \int_{\Omega} G_{k}(x-y)\beta(y) u(y) \,dy 
      = u^{\rm inc}(x)
\end{equation}
where we use the notation 
$$
  u^{\rm inc}(x):=-\int G_{k}(x-y)f(y)\,dy\;.
$$
The fact that the coefficients $\alpha$ and $\beta$ vanish outside of $\Omega$ permits to consider the integral equation \eqref{e:vie} on any domain $\widehat\Omega$ satisfying $\Omega\subset\widehat\Omega\subset\R^{d}$. Once $u$ solves \eqref{e:vie} on $\widehat\Omega$, one can use the same formula \eqref{e:vie} to extend $u$ outside of $\widehat\Omega$. It is clear that the resulting function $u$ will not depend on $\widehat\Omega$ and will be a solution of the original scattering problem \eqref{e:scat}. In the following we will make the minimal choice $\widehat\Omega=\Omega$ and therefore consider \eqref{e:vie} as an integral equation on $\Omega$. 
We shall abbreviate this integral equation as
\begin{equation}
\label{e:Au=f}
 u - Au = u^{\rm inc}
\end{equation}
with
\begin{equation}
\label{e:A}
 Au(x) = \Div\int_{\Omega}G_{k}(x-y)\alpha(y)\grad u(y)\, dy
      +   \int_{\Omega} G_{k}(x-y)\beta(y) u(y) \,dy \,.
\end{equation}
Assuming that $\Omega$ is a bounded Lipschitz domain, one can consider the VIE \eqref{e:Au=f} in the standard Sobolev spaces $H^{s}(\Omega)$. The natural energy space associated with the second order PDE \eqref{e:scat} is $H^{1}(\Omega)$, but other values of $s$ can be interesting, too, in particular $s=0$, i.e. the space $L^{2}(\Omega)$, which  
seems naturally associated with the apparent structure of \eqref{e:Au=f} as a second kind integral equation and may be useful for analyzing certain numerical algorithms for its solution.

The convolution with $G_{k}$ is a pseudodifferential operator of order $-2$, mapping distributions with compact support and Sobolev regularity $s$ to $H^{s+2}_{\rm loc}(\R^{d})$ for any $s\in\R$, which implies immediately boundedness of the operator $A$ in low order Sobolev spaces:
\begin{proposition}
\label{p:bddL2}
Let $\alpha,\beta\in L^{\infty}(\Omega)$. Then 
$$
  A: H^{1}(\Omega) \to H^{1}(\Omega) \; \mbox{ is bounded }.
$$
If in addition $\grad \alpha\in L^{\infty}(\Omega)$, then $A$ is a bounded operator in $L^{2}(\Omega)$.
\end{proposition}
Another immediate observation is that the second integral operator in \eqref{e:A} maps $L^{2}$ to $H^{2}$, and is therefore compact as an operator in $L^{2}$ and in $H^{1}$. This is relevant if $a(x)$ is constant everywhere, since then $\alpha\equiv0$ and the first integral operator in \eqref{e:A}, which is not compact, in general, is absent.
\begin{theorem}
\label{t:a=1}
Let $a(x)=1$ in $\R^{d}$ and $k\in L^{\infty}(\R^{d})$. Then the VIE \eqref{e:vie} is a second kind Fredholm integral equation with a weakly singular kernel and the Fredholm alternative holds: The operator $\Id -A$ is a Fredholm operator of index zero in $L^{2}(\Omega)$ and in $H^{1}(\Omega)$. 
\end{theorem}

\section{Smooth Coefficients}
Besides the case of the Laplace operator addressed in Theorem~\ref{t:a=1}, another situation is well known and is studied for example in the book \cite{CoKr83}. This is the case of a coefficient $a(x)$ that is smooth on all of $\R^{d}$. In this case, $\alpha=0$ on the boundary $\Gamma=\partial\Omega$, and the first integral operator in \eqref{e:A} can be transformed by integration by parts.
\begin{multline*}
 \Div G_{k}*(\alpha\grad u)(x) = 
 -\Div\int_{\Omega}\nabla_{y}\big(G_{k}(x-y)\alpha(y)\big)\,u(y)\,dy\\
 = \Delta\int_{\Omega}G_{k}(x-y)\alpha(y)u(y)\,dy 
 -\Div\int_{\Omega}G_{k}(x-y)(\grad\alpha)(y)\,u(y)\,dy\\
 = -\alpha(x)u(x)  -k^{2}\!\!\int_{\Omega}G_{k}(x-y)\alpha(y)u(y)\,dy 
 -\Div\!\int_{\Omega}G_{k}(x-y)(\grad\alpha)(y)\,u(y)\,dy.
\end{multline*}
This allows to write the VIE \eqref{e:vie} in an equivalent form that shows its nature as a Fredholm integral equation of the second kind with a weakly singular kernel.
\begin{equation}
\label{e:viesmooth}
\begin{split}
 &a(x)u(x) 
      -   \int_{\Omega} G_{k}(x-y)(\beta(y)-k^{2}\alpha(y)) u(y) \,dy\\
       &\quad+\Div\int_{\Omega}G_{k}(x-y)(\grad\alpha)(y)\,u(y)\,dy
        - \int_{\Omega} G_{k}(x-y)\beta(y) u(y) \,dy
      = u^{\rm inc}(x) 
\end{split}
\end{equation}
\begin{theorem}
\label{t:asmooth}
Let $a\in C^{1}(\R^{d})$ and $k\in L^{\infty}(\R^{d})$. Then the operator $\Id -A$ is a Fredholm operator of index zero in $L^{2}(\Omega)$ and in $H^{1}(\Omega)$ if and only if $a(x)\ne0$ for all $x\in\Omega$.
\end{theorem}

\section{Piecewise Smooth Coefficients}
In obstacle scattering, the case of a globally smooth coefficient $a(x)$ is not natural. There one expects rather a sharp interface where the material properties change discontinuously. We thus assume that the coefficient $a$ is piecewise $C^{1}$, which means that $\alpha\in C^{1}(\overline\Omega)$.

One can then still carry out the partial integration as in the previous section, but there will appear an additional term on the boundary $\Gamma=\partial\Omega$:
\begin{multline*}
 \Div G_{k}*(\alpha\grad u)(x) \\
 = -\Div\int_{\Omega}\nabla_{y}\big(G_{k}(x-y)\alpha(y)\big)\,u(y)\,dy
   + \Div\int_{\Gamma}n(y)G_{k}(x-y)\alpha(y)u(y)\,ds(y)\\
 = -\alpha(x)u(x)  -k^{2}\int_{\Omega}G_{k}(x-y)\alpha(y)u(y)\,dy 
     -\Div\int_{\Omega}G_{k}(x-y)(\grad\alpha)(y)\,u(y)\,dy\\
   - \int_{\Gamma}\partial_{n(y)}G_{k}(x-y)\alpha(y)u(y)\,ds(y)\,.
\end{multline*}
The additional term is just the Helmholtz double layer potential with density $\alpha u$, which we can abbreviate as
$
  \mathscr{D}\gamma(\alpha u)\,.
$
Here $\gamma:H^{1}(\Omega) \to H^{\frac12}(\Gamma)$ is the trace operator. We obtain our volume integral operator in the form
\begin{equation}
\label{e:A1+D}
 (\Id - A)u(x) = a(x) u(x) + A_{1}u(x) + \mathscr{D}\gamma(\alpha u)(x)
\end{equation}
with
\begin{multline*}
  A_{1} u(x) = 
  -k^{2}\int_{\Omega}G_{k}(x-y)\alpha(y)u(y)\,dy \\
     +\Div\int_{\Omega}G_{k}(x-y)(\grad\alpha)(y)\,u(y)\,dy
     - \int_{\Omega} G_{k}(x-y)\beta(y) u(y) \,dy\,.
\end{multline*}
The operator $A_{1}$ is bounded from $L^{2}(\Omega)$ to $H^{1}(\Omega)$, hence compact as an operator in $H^{1}(\Omega)$. 

The operator $u\mapsto \mathscr{D}\gamma(\alpha u)$ is bounded in $H^{1}(\Omega)$ but not compact, in general. It is also not continuous with respect to the $L^{2}(\Omega)$-norm of $u$. This implies that the operator $\Id-A$, despite being generated from a pseudodifferential operator of order zero, does not have a continuous extension to $L^{2}(\Omega)$ from the dense subspace $H^{1}(\Omega)$. It does have a continuous extension to $L^{2}(\Omega)$ from the subspace $H^{1}_{0}(\Omega)$, but this is a different operator, where the last term in \eqref{e:A1+D} is missing. 

\subsection{Extension to a Boundary-Domain System}
From the VIE \eqref{e:Au=f} with the integral operator written in the form \eqref{e:A1+D}, we can get an equation on the boundary by taking the trace on $\Gamma$:
\begin{equation}
\label{e:onGamma}
 a\gamma  u + \gamma A_{1}u + \gamma \mathscr{D}\gamma(\alpha u) = \gamma u^{\rm inc}\,.
\end{equation}
We now treat the trace $\gamma u$ as if it was an additional unknown, denoted by $\phi$, and consider the two equations \eqref{e:Au=f} and \eqref{e:onGamma} as a coupled boundary-domain integral equation system.

Taking into account the jump relation for the double layer potential
$$
  \gamma\mathscr{D}\phi = -\tfrac12 \phi + K\phi,
$$
where $K$ is the Helmholtz double layer potential operator evaluated on $\Gamma$, as well as the fact that the commutator $[K,\alpha]$ between $K$ and the multiplication by $\alpha$ is compact in the trace space $H^{\frac12}(\Gamma)$, we can write this coupled system in the following matrix form.
\begin{equation}
\label{e:bdry-domain}
 \begin{pmatrix}
   a\Id+A_{1} & \mathscr{D}(\gamma\alpha\cdot)\\
   \gamma A_{1} &\frac12(1+a)\Id +\alpha K + [K,\alpha]
 \end{pmatrix}
 \begin{pmatrix} u\\ \phi \end{pmatrix}
 =
  \begin{pmatrix} u^{\rm inc}\\ \psi \end{pmatrix}
\end{equation}
It is easy to see that this system is equivalent to the original VIE in the following sense.
\begin{proposition}
\label{p:equiv}
 Let $\Omega$ be a bounded Lipschitz domain with boundary $\Gamma$. 
 Let $\alpha\in C^{1}(\overline\Omega)$ and $\beta\in L^{\infty}(\Omega)$,
 and let $u^{\rm inc}\in H^{1}(\Omega)$ be given.\\
 If $u\in H^{1}(\Omega)$ is a solution of the VIE \eqref{e:Au=f}, then 
 $\begin{pmatrix}u\\\phi\end{pmatrix}=\begin{pmatrix}u\\\gamma u\end{pmatrix}$
 solves the coupled system \eqref{e:bdry-domain} with $\psi=\gamma u^{\rm inc}$.\\
 Conversely, let $\psi\in H^{\frac12}(\Gamma)$ be given and 
 $\begin{pmatrix}u\\\phi\end{pmatrix}\in H^{1}(\Omega)\times H^{\frac12}(\Gamma)$ be a solution of the coupled system \eqref{e:bdry-domain}. 
If $\psi=\gamma u^{\rm inc}$, and if $\gamma a\ne0$ a.e. on 
$\Gamma$, then $\phi=\gamma u$, and $u$ is a solution of the VIE \eqref{e:Au=f}.
\end{proposition}
\begin{proof}
The construction of the coupled system shows that it is satisfied by any solution of the VIE and its trace on the boundary. To show the converse, one subtracts the trace of the first equation in \eqref{e:bdry-domain} from the second and finds
$$
   \gamma a\,\big(\gamma u - \phi\big)=0.
$$
Since we assume that $\gamma a$ does not vanish on a set of positive measure, $\phi=\gamma u$ follows.
\end{proof}

\subsection{Lipschitz Boundary}
The system \eqref{e:bdry-domain} is easier to analyze than the original VIE \eqref{e:Au=f}. This is due to the fact that now the main difficulty is pushed to the boundary integral operator $K$, which is a well-studied classical boundary integral operator \cite{CoLip}. Indeed, splitting off the operators that we already have identified as compact operators, and taking into account that the coupling operator
$ \phi\mapsto \mathscr{D}(\gamma\alpha\phi) $
is bounded from $H^{\frac12}(\Gamma)$ to $H^{1}(\Gamma)$ \cite{CoLip}, we see that the Fredholm alternative holds for the system \eqref{e:bdry-domain} (and therefore for the VIE \eqref{e:Au=f}) if and only if the operator
$$
  \widehat A = 
  \begin{pmatrix}
     a\Id & \mathscr{D}(\gamma\alpha\cdot)\\
     0 &\frac12(1+a)\Id +\alpha K
 \end{pmatrix}
$$
is a Fredholm operator of index zero in the space $H^{1}(\Omega)\times H^{\frac12}(\Gamma)$. This, in turn, is the case if and only if both
$$
 a\Id:H^{1}(\Omega)\to H^{1}(\Omega)\quad\mbox{ and }\quad  
 \frac12(1+a)\Id +\alpha K: H^{\frac12}(\Gamma)\to H^{\frac12}(\Gamma)
$$
are Fredholm of index zero. 
We have shown the following result.
\begin{theorem}
\label{t:lip}
 Let $\Omega$ be a bounded Lipschitz domain with boundary $\Gamma$. 
 Let $\alpha\in C^{1}(\overline\Omega)$ and $\beta\in L^{\infty}(\Omega)$.
Then for the VIE \eqref{e:vie} the Fredholm alternative holds in $H^{1}(\Omega)$ if and only if 
\begin{itemize}
\item[\emph{(i) }] $a(x)\ne0$ in $\overline\Omega$ and
\item[\emph{(ii)}] $\frac12(1+a)\Id +\alpha K$ is Fredholm of index zero in $H^{\frac12}(\Gamma)$.
\end{itemize}
\end{theorem}
Condition (ii) can be made more precise by using information about the essential spectrum of the operator $\frac12\Id+K$. This operator differs by a compact operator from the corresponding operator for $k=0$, i.e. the trace of the harmonic double layer potential operator. The latter is known to be a positive selfadjoint contraction in $H^{\frac12}(\Gamma)$ if this space is equipped with a suitable scalar product, see \cite{CoRemPos}. 

Therefore its essential spectrum, which is also the essential spectrum of the operator $\frac12\Id+K$, is a compact subset $\Sigma$ of the open interval $(0,1)$. It is known that for any Lipschitz boundary $\frac12\in\Sigma$, that for smooth boundaries 
$\Sigma=\{\frac12\}$, and that for polygons in $\R^{2}$, $\Sigma$ is an interval depending on the corner angles.

If the coefficient function $a$ is piecewise constant, so that $\alpha=a-1$ is a constant on $\Gamma$, the operator $\frac12(1+a)\Id +\alpha K$ is either the identity if $\alpha=0$ or a multiple of the operator $\sigma\Id -(\frac12\Id+K)$ with
\begin{equation}
\label{e:alphasigma}
  \frac{1+a}{2(1-a)}=\sigma-\frac12 \quad\Longleftrightarrow\quad
  a=\frac{\sigma-1}\sigma\,.
\end{equation}
It follows that the operator $\frac12(1+a)\Id +\alpha K$ is Fredholm of index zero if and only if $\sigma\ne\Sigma$.

If the function $\alpha$ is not constant on $\Gamma$, one can use the fact that the operator $K$ commutes modulo compact operators with multiplications by $C^{1}$ functions and apply standard localization procedures. The result is that if for each point $x\in\Gamma$, the number $\sigma$ from \eqref{e:alphasigma} does not belong to the essential spectrum $\Sigma$, then the operator $\frac12(1+a)\Id +\alpha K$ is Fredholm. This condition
\begin{equation}
\label{e:nec}
 \forall\, x\in\Gamma: \quad \frac{1}{1-a(x)} \not\in \Sigma 
\end{equation}
is, in general, only a sufficient condition. In order to obtain a necessary condition, one would need a ``localized'' version $\Sigma_{x}$ of $\Sigma$, which is only known in some cases, namely when $\Gamma$ has a suitable tangent cone at $x$.

We summarize this discussion.
\begin{theorem}
\label{t:lipprec}
Assume the hypotheses of Theorem \ref{t:lip}. Let $\Sigma\subset(0,1)$ be the essential spectrum of the operator $\frac12\Id+K$ in $H^{\frac12}(\Gamma)$.
If the coefficient $a\in C^{1}(\overline\Omega)$ is constant on $\Gamma$, then the volume integral operator $\Id-A$ is Fredholm of index zero in $H^{1}(\Omega)$ if and only if
\begin{itemize}
\item[\emph{(i) }] $a(x)\ne0$ in $\overline\Omega$ and
\item[\emph{(ii)}] $a(x)\ne \frac{\sigma-1}\sigma\quad \mbox{ for } x\in \Gamma,\; \sigma\in\Sigma \,.$
\end{itemize}
If $a$ is not constant on $\Gamma$, then the conditions \emph{(i)} and \emph{(ii)} imply that the volume integral operator is Fredholm in $H^{1}(\Omega)$. 
\end{theorem}

\subsection{Smooth Boundary}
If $\Gamma$ is smooth ($C^{1+\epsilon}$ with $\epsilon>0$), then the boundary integral operator $K$ has a weakly singular kernel and is compact in $H^{\frac12}(\Gamma)$. This implies that $\Sigma=\{\frac12\}$ in Theorem~\ref{t:lipprec}. But it also implies directly that the operator $\frac12(1+a)\Id +\alpha K$ is Fredholm of index zero if and only if $1+a$ does not vanish. We obtain immediately as a corollary of Theorem~\ref{t:lip} the following result.
\begin{theorem}
\label{t:smooth}
 Let $\Omega$ be a bounded smooth (Lyapunov) domain. 
 Let $\alpha\in C^{1}(\overline\Omega)$ and $\beta\in L^{\infty}(\Omega)$.
Then for the VIE \eqref{e:vie} the Fredholm alternative holds in $H^{1}(\Omega)$ if and only if 
\begin{itemize}
\item[\emph{(i) }] $a(x)\ne0$ in $\overline\Omega$ and
\item[\emph{(ii)}] $a(x)\ne-1$ on $\Gamma$.
\end{itemize}
\end{theorem}
The conditions on the coefficient $a(x)$ obtained in Theorem~\ref{t:smooth} have been known for a long time as conditions for Fredholm properties of the scattering problem \eqref{e:scat}. In \cite{CoSte85}, the case of piecewise constant coefficients was treated. Using the method of boundary integral equations, the case of smooth boundaries in any dimension and the case of polygons in dimension two was studied. In the thesis \cite{ChesnelThesis2012} and the paper \cite{Chesneletal2012}, variational methods for the interface problem were used to obtain the same conditions in the case of smooth domains and also necessary and sufficient conditions for some non-smooth domains.


\end{document}